\documentclass[12pt]{article}
\usepackage{amsmath,amssymb}    

\makeatletter
\newfont{\footsc}{cmcsc10 at 8truept}
\newfont{\footbf}{cmbx10 at 8truept}  
\newfont{\footrm}{cmr10 at 10truept} 
\makeatother
\def\ang#1{\left\langle#1\right\rangle}

\newtheorem{theorem}{Theorem}

\newtheorem{Lem}{Lemma}

\newcommand{\qed}{\hfill \rule{0.7ex}{1.5ex}}

\newenvironment{proof}{\begin{trivlist} \item[\hskip \labelsep{\it
Proof.}]\setlength{\parindent}{0pt}}{\end{trivlist}}

\title{Barnes' multiple Bernoulli and poly-Bernoulli mixed-type polynomials}

\author{
Dae San Kim 
\thanks{
This work was supported by the National Research Foundation of Korea (NRF) grant funded by the Korean government (MOE) (No.2012R1A1A2003786). 
}\\
\small Department of Mathematics, Sogang University\\[-0.8ex]
\small Seoul 121-741, Republic of Korea\\[-0.8ex]
\small \texttt{dskim@sogang.ac.kr}\\\\
Taekyun Kim \\
\small Department of Mathematics, Kwangwoon University\\[-0.8ex]
\small Seoul 139-701, Republic of Korea\\[-0.8ex]
\small \texttt{tkkim@kw.ac.kr}\\\\
Takao Komatsu
\thanks{
The third author was supported in part by the Grant-in-Aid for Scientific research (C) (No.22540005), the Japan Society for the Promotion of Science. 
}\\
\small Graduate School of Science and Technology, Hirosaki University\\[-0.8ex]
\small Hirosaki 036-8561, Japan\\[-0.8ex]
\small \texttt{komatsu@cc.hirosaki-u.ac.jp}}

\date{
\small MR Subject Classifications: 05A15, 05A40, 11B68, 11B75, 65Q05}

\begin{document}
\maketitle

\begin{abstract}
In this paper, we consider Barnes' multiple Bernoulli and poly-Bernoulli mixed-type polynomials.  From the properties of Sheffer sequences of these polynomials arising from umbral calculus, we derive new and interesting identities. 
\end{abstract}

\section{Introduction}

In this paper, we consider the polynomials $S_n^{(r,k)}(x|a_1,\dots,a_r)$ whose generating function is given by 
\begin{equation}  
\frac{t^r}{\prod_{j=1}^r(e^{a_j t}-1)}\frac{{\rm Li}_k(1-e^{-t})}{1-e^{-t}}e^{x t}=\sum_{n=0}^\infty S_n^{(r,k)}(x|a_1,\dots,a_r)\frac{t^n}{n!}\,,
\label{barnesmpbp}
\end{equation} 
where $r\in\mathbb Z_{>0}$, $k\in\mathbb Z$, $a_1,\dots,a_r\ne 0$, and 
$$
{\rm Li}_k(x)=\sum_{m=1}^\infty\frac{x^m}{m^k}
$$ 
is the $k$th polylogarithm function. 
$S_n^{(r,k)}(x|a_1,\dots,a_r)$ will be called Barnes' multiple Bernoulli and poly-Bernoulli mixed-type polynomials. 
When $S_n^{(r,k)}(a_1,\dots,a_r)$\\$=S_n^{(r,k)}(0|a_1,\dots,a_r)$ will be called Barnes' multiple Bernoulli and poly-Bernoulli mixed-type numbers. 

Recall that, for every integer $k$, the poly-Bernoulli polynomials $B_n^{(k)}(x)$ are defined by the generating function as 
\begin{equation} 
\frac{{\rm Li}_k(1-e^{-t})}{1-e^{-t}}e^{x t}=\sum_{n=0}^\infty B_n^{(k)}(x)\frac{t^n}{n!} 
\label{pb}
\end{equation}  
(\cite{BH}, {\it Cf.}\cite{CC}). Also, recall that the Barnes' multiple Bernoulli polynomials $B_n(x|a_1,\dots,a_r)$ are defined by the generating function as 
\begin{equation} 
\frac{t^r}{\prod_{j=1}^r(e^{a_j t}-1)}e^{x t}=\sum_{n=0}^\infty B_n(x|a_1,\dots,a_r)\frac{t^n}{n!}\,,
\label{bmb} 
\end{equation}  
where $a_1,\dots,a_r\ne 0$ (\cite{JKKLPR}).  

In this paper, we consider Barnes' multiple Bernoulli and poly-Bernoulli mixed-type polynomials.  From the properties of Sheffer sequences of these polynomials arising from umbral calculus, we derive new and interesting identities.

\section{Umbral calculus}

Let $\mathbb C$ be the complex number field and let $\mathcal F$ be the set of all formal power series in the variable $t$: 
\begin{equation} 
\mathcal F=\left\{f(t)=\sum_{k=0}^\infty\frac{a_k}{k!}t^k\Bigg|a_k\in\mathbb C\right\}\,.
\label{uc:fps}
\end{equation} 

Let $\mathbb P=\mathbb C[x]$ and let $\mathbb P^\ast$ be the vector space of all linear functionals on $\mathbb P$. 
$\ang{L|p(x)}$ is the action of the linear functional $L$ on the polynomial $p(x)$, and
we recall that the vector space operations on $\mathbb P^\ast$ are defined by $\ang{L+M|p(x)}=\ang{L|p(x)}+\ang{M|p(x)}$, $\ang{cL|p(x)}=c\ang{L|p(x)}$, where $c$ is a complex constant in $\mathbb C$. For $f(t)\in\mathcal F$, let us define the linear functional on $\mathbb P$ by setting
\begin{equation} 
\ang{f(t)|x^n}=a_n,\quad (n\ge 0). 
\label{uc9}
\end{equation} 
In particular, 
\begin{equation} 
\ang{t^k|x^n}=n!\delta_{n,k}\quad (n, k\ge 0),
\label{uc10}
\end{equation} 
where $\delta_{n,k}$ is the Kronecker's symbol. 

For $f_L(t)=\sum_{k=0}^\infty\frac{\ang{L|x^k}}{k!}t^k$, we have $\ang{f_L(t)|x^n}=\ang{L|x^n}$. That is, $L=f_L(t)$. 
The map $L\mapsto f_L(t)$ is a vector space isomorphism from $\mathbb P^\ast$ onto $\mathcal F$. Henceforth, $\mathcal F$ denotes both the algebra of formal power series in $t$ and the vector space of all linear functionals on $\mathbb P$, and so an element $f(t)$ of $\mathcal F$ will be thought of as both a formal power series and a linear functional. We call $\mathcal F$ the {\it umbral algebra} and the {\it umbral calculus} is the study of umbral algebra. 
The order $O\bigl(f(t)\bigr)$ of a power series $f(t)(\ne 0)$ is the smallest integer $k$ for which the coefficient of $t^k$ does not vanish. If $O\bigl(f(t)\bigr)=1$, then $f(t)$ is called a {\it delta series}; if $O\bigl(f(t)\bigr)=0$, then $f(t)$ is called an {\it invertible series}. 
For $f(t), g(t)\in\mathcal F$ with $O\bigl(f(t)\bigr)=1$ and $O\bigl(g(t)\bigr)=0$, there exists a unique sequence $s_n(x)$ ($\deg s_n(x)=n$) such that $\ang{g(t)f(t)^k|s_n(x)}=n!\delta_{n,k}$, for $n,k\ge 0$. 
Such a sequence $s_n(x)$ is called the {\it Sheffer sequence} for $\bigl(g(t), f(t)\bigr)$ which is denoted by $s_n(x)\sim\bigl(g(t), f(t)\bigr)$. 

For $f(t), g(t)\in\mathcal F$ and $p(x)\in\mathbb P$, we have 
\begin{equation} 
\ang{f(t)g(t)|p(x)}=\ang{f(t)|g(t)p(x)}=\ang{g(t)|f(t)p(x)} 
\label{uc11} 
\end{equation} 
and
\begin{equation} 
f(t)=\sum_{k=0}^\infty\ang{f(t)|x^k}\frac{t^k}{k!},\quad 
p(x)=\sum_{k=0}^\infty\ang{t^k|p(x)}\frac{x^k}{k!}
\label{uc12}
\end{equation} 
(\cite[Theorem 2.2.5]{Roman}). Thus, by (\ref{uc12}), we get
\begin{equation} 
t^k p(x)=p^{(k)}(x)=\frac{d^k p(x)}{d x^k}\quad\hbox{and}\quad e^{yt}p(x)=p(x+y).
\label{uc13} 
\end{equation} 

Sheffer sequences are characterized in the generating function (\cite[Theorem 2.3.4]{Roman}). 

\begin{Lem} 
The sequence $s_n(x)$ is Sheffer for $\big(g(t),f(t)\bigr)$ if and only if 
$$
\frac{1}{g\bigl(\bar f(t)\bigr)}e^{y\bar f(t)}=\sum_{k=0}^\infty\frac{s_k(y)}{k!}t^k\quad(y\in\mathbb C)\,,
$$
where $\bar f(t)$ is the compositional inverse of $f(t)$.  
\label{th234}
\end{Lem}  

For $s_n(x)\sim\bigl(g(t), f(t)\bigr)$, we have the following equations (\cite[Theorem 2.3.7, Theorem 2.3.5, Theorem 2.3.9]{Roman}): 
\begin{align} 
f(t)s_n(x)&=n s_{n-1}(x)\quad (n\ge 0), 
\label{uc15}\\ 
s_n(x)&=\sum_{j=0}^n\frac{1}{j!}\ang{g\bigl(\bar f(t)\bigr)^{-1}\bar f(t)^j|x^n}x^j,
\label{uc16}\\
s_n(x+y)&=\sum_{j=0}^n\binom{n}{j}s_j(x)p_{n-j}(y)\,, 
\label{uc17}
\end{align} 
where $p_n(x)=g(t)s_n(x)$.

Assume that $p_n(x)\sim\bigl(1, f(t)\bigr)$ and $q_n(x)\sim\bigl(1, g(t)\bigr)$. Then the transfer formula (\cite[Corollary 3.8.2]{Roman}) is given by 
$$ 
q_n(x)=x\left(\frac{f(t)}{g(t)}\right)^n x^{-1}p_n(x)\quad (n\ge 1). 
$$  
For $s_n(x)\sim\bigl(g(t), f(t)\bigr)$ and $r_n(x)\sim\bigl(h(t), l(t)\bigr)$, assume that 
$$
s_n(x)=\sum_{m=0}^n C_{n,m}r_m(x)\quad (n\ge 0)\,. 
$$ 
Then we have (\cite[p.132]{Roman})  
\begin{equation} 
C_{n,m}=\frac{1}{m!}\ang{\frac{h\bigl(\bar f(t)\bigr)}{g\bigl(\bar f(t)\bigr)}l\bigl(\bar f(t)\bigr)^m\Bigg|x^n}\,. 
\label{uc19}
\end{equation}

\section{Main results}  

We now note that $B_n^{(k)}(x)$, $B_n(x|a_1,\dots,a_r)$ and $S_n^{(r,k)}(x|a_1,\dots,a_r)$ are the Appell sequences for 
$$
g_k(t)=\frac{1-e^{-t}}{{\rm Li}_k(1-e^{-t})},\quad g_r(t)=\frac{\prod_{j=1}^r(e^{a_j t}-1)}{t^r},\quad g_{r,k}(t)=\frac{\prod_{j=1}^r(e^{a_j t}-1)}{t^r}\frac{1-e^{-t}}{{\rm Li}_k(1-e^{-t})}\,. 
$$ 
So, 
\begin{align} 
B_n^{(k)}(x)&\sim\left(\frac{1-e^{-t}}{{\rm Li}_k(1-e^{-t})},t\right)\,,
\label{10a}\\
B_n(x|a_1,\dots,a_r)&\sim\left(\frac{\prod_{j=1}^r(e^{a_j t}-1)}{t^r}, t\right)\,,\label{10b}\\
S_n^{(r,k)}(x|a_1,\dots,a_r)&\sim\left(\frac{\prod_{j=1}^r(e^{a_j t}-1)}{t^r}\frac{1-e^{-t}}{{\rm Li}_k(1-e^{-t})},t\right)\,. 
\label{10c} 
\end{align}  
In particular, we have  
\begin{align} 
t B_n^{(k)}(x)&=\frac{d}{dx}B_n^{(k)}(x)=n B_{n-1}^{(k)}(x)\,,
\label{11a}\\
t B_n(x|a_1,\dots,a_r)&=\frac{d}{dx}B_n(x|a_1,\dots,a_r)\notag\\
&=n B_{n-1}(x|a_1,\dots,a_r)\,,\label{11b}\\
t S_n^{(r,k)}(x|a_1,\dots,a_r)&=\frac{d}{dx}S_n^{(r,k)}(x|a_1,\dots,a_r)\notag\\
&=n S_{n-1}^{(r,k)}(x|a_1,\dots,a_r)\,. 
\label{11c} 
\end{align}  
Notice that  
$$
\frac{d}{dx}{\rm Li}_k(x)=\frac{1}{x}{\rm Li}_{k-1}(x)\,. 
$$ 


\subsection{Explicit expressions}

Write $B_n(a_1,\dots,a_r):=B_n(0|a_1,\dots,a_r)$ and $S_n^{(r,k)}(a_1,\dots,a_r):=S_n^{(r,k)}(0|a_1,\dots,a_r)$.  Let $(n)_j=n(n-1)\cdots(n-j+1)$ ($j\ge 1$) with $(n)_0=1$.

\begin{theorem} 
\begin{align} 
S_n^{(r,k)}(x|a_1,\dots,a_r)&=\sum_{l=0}^n\binom{n}{l}B_{n-l}(a_1,\dots,a_r)B_l^{(k)}(x)\,,\label{100a}\\  
&=\sum_{l=0}^n\binom{n}{l}B_{n-l}^{(k)}B_l(x|a_1,\dots,a_r)\,,
\label{100b}\\
&=\sum_{l=0}^n\sum_{m=l}^n\sum_{j=0}^m(-1)^j\binom{m}{j}\binom{n}{l}\frac{1}{(m+1)^k}B_{n-l}(a_1,\dots,a_r)(x-j)^l\,, 
\label{100c}\\
&=\sum_{l=0}^n\left(\sum_{j=l}^n\sum_{m=0}^{n-j}(-1)^{n-m-j}\binom{n}{j}\binom{j}{l}\right.\notag\\&\qquad\qquad\left.\times
\frac{m!}{(m+1)^k}S_2(n-j,m)B_{j-l}(a_1,\dots,a_r)\right)x^l\,,  
\label{100d}\\ 
&=\sum_{j=0}^n\binom{n}{j}S_{n-j}^{(r,k)}(a_1,\dots,a_r)x^j\,. 
\label{100e} 
\end{align}  
\label{th100}  
\end{theorem} 

\begin{proof} 
By (\ref{barnesmpbp}), (\ref{pb}) and (\ref{bmb}), we have 
\begin{align*}  
S_n^{(r,k)}(y|a_1,\dots,a_r)&=\ang{\sum_{i=0}^\infty S_i^{(r,k)}(y|a_1,\dots,a_r)\frac{t^i}{i!}\Big| x^n}\\
&=\ang{\frac{t^r}{\prod_{j=1}^r(e^{a_j t}-1)}\frac{{\rm Li}_k(1-e^{-t})}{1-e^{-t}}e^{yt}\Big| x^n}\\
&=\ang{\frac{t^r}{\prod_{j=1}^r(e^{a_j t}-1)}\Big|\frac{{\rm Li}_k(1-e^{-t})}{1-e^{-t}}e^{yt}x^n}\\
&=\ang{\frac{t^r}{\prod_{j=1}^r(e^{a_j t}-1)}\Big|\sum_{l=0}^\infty B_l^{(k)}(y)\frac{t^l}{l!}x^n}\\
&=\ang{\frac{t^r}{\prod_{j=1}^r(e^{a_j t}-1)}\Big|\sum_{l=0}^n\binom{n}{l}B_l^{(k)}(y)x^{n-l}}\\
&=\sum_{l=0}^n\binom{n}{l}B_l^{(k)}(y)\ang{\frac{t^r}{\prod_{j=1}^r(e^{a_j t}-1)}\Big|x^{n-l}}\\
&=\sum_{l=0}^n\binom{n}{l}B_l^{(k)}(y)\ang{\sum_{i=0}^\infty B_i(a_1,\dots,a_r)\frac{t^i}{i!}\Big|x^{n-l}}\\
&=\sum_{l=0}^n\binom{n}{l}B_l^{(k)}(y)B_{n-l}(a_1,\dots,a_r)\,.  
\end{align*} 
So, we get (\ref{100a}).


We also have 
\begin{align*}  
S_n^{(r,k)}(y|a_1,\dots,a_r)&=\ang{\sum_{i=0}^\infty S_i^{(r,k)}(y|a_1,\dots,a_r)\frac{t^i}{i!}\Big| x^n}\\
&=\ang{\frac{{\rm Li}_k(1-e^{-t})}{1-e^{-t}}\Big|\frac{t^r}{\prod_{j=1}^r(e^{a_j t}-1)}e^{yt}x^n}\\
&=\ang{\frac{{\rm Li}_k(1-e^{-t})}{1-e^{-t}}\Big|\sum_{l=0}^\infty B_l(y|a_1,\dots,a_r)\frac{t^l}{l!}x^n}\\
&=\ang{\frac{{\rm Li}_k(1-e^{-t})}{1-e^{-t}}\Big|\sum_{l=0}^n B_l(y|a_1,\dots,a_r)\binom{n}{l}x^{n-l}}\\
&=\sum_{l=0}^n\binom{n}{l}B_l(y|a_1,\dots,a_r)\ang{\frac{{\rm Li}_k(1-e^{-t})}{1-e^{-t}}\Big|x^{n-l}}\\
&=\sum_{l=0}^n\binom{n}{l}B_l(y|a_1,\dots,a_r)\ang{\sum_{i=0}^\infty B_i^{(k)}\frac{t^i}{i!}\Big|x^{n-l}}\\
&=\sum_{l=0}^n\binom{n}{l}B_l(y|a_1,\dots,a_r)B_{n-l}^{(k)}\,. 
\end{align*} 
Thus, we get (\ref{100b}).   


In \cite{KKL} we obtained that 
$$
\frac{{\rm Li}_k(1-e^{-t})}{1-e^{-t}}x^n=\sum_{m=0}^n\frac{1}{(m+1)^k}\sum_{j=0}^m(-1)^j\binom{m}{j}(x-j)^n\,. 
$$ 
So, 
\begin{align*}  
S_n^{(r,k)}(x|a_1,\dots,a_r)&=\frac{t^r}{\prod_{j=1}^r(e^{a_j t}-1)}\frac{{\rm Li}_k(1-e^{-t})}{1-e^{-t}}x^n\\
&=\sum_{m=0}^n\frac{1}{(m+1)^k}\sum_{j=0}^m(-1)^j\binom{m}{j}\frac{t^r}{\prod_{j=1}^r(e^{a_j t}-1)}(x-j)^n\\
&=\sum_{m=0}^n\frac{1}{(m+1)^k}\sum_{j=0}^m(-1)^j\binom{m}{j}\sum_{l=0}^n\binom{n}{l}B_{n-l}(a_1,\dots,a_r)(x-j)^l\\
&=\sum_{l=0}^n\sum_{m=l}^n\sum_{j=0}^m(-1)^j\binom{m}{j}\binom{n}{l}\frac{1}{(m+1)^k}B_{n-l}(a_1,\dots,a_r)(x-j)^l\,,   
\end{align*} 
which is the identity (\ref{100c}).


In \cite{KKL} we obtained that 
$$
\frac{{\rm Li}_k(1-e^{-t})}{1-e^{-t}}x^n=\sum_{j=0}^n\left(\sum_{m=0}^{n-j}\frac{(-1)^{n-m-j}}{(m+1)^k}\binom{n}{j}m! S_2(n-j,m)\right)x^j\,,  
$$ 
where $S_2(l,m)$ are the Stirling numbers of the second kind, defined by 
$$
(e^t-1)^m
=m!\sum_{l=m}^\infty S_2(l,m)\frac{t^l}{l!}\,. 
$$ 
Thus, 
\begin{align*} 
&S_n^{(r,k)}(x|a_1,\dots,a_r)=\sum_{j=0}^n\left(\sum_{m=0}^{n-j}\frac{(-1)^{n-m-j}}{(m+1)^k}\binom{n}{j}m! S_2(n-j,m)\right)\frac{t^r}{\prod_{i=1}^r(e^{a_i t}-1)}x^j\\
&=\sum_{j=0}^n\left(\sum_{m=0}^{n-j}\frac{(-1)^{n-m-j}}{(m+1)^k}\binom{n}{j}m! S_2(n-j,m)\right)B_j(x|a_1,\dots,a_r)\\
&=\sum_{j=0}^n\left(\sum_{m=0}^{n-j}\frac{(-1)^{n-m-j}}{(m+1)^k}\binom{n}{j}m! S_2(n-j,m)\right)\sum_{l=0}^j\binom{j}{l}B_{j-l}(a_1,\dots,a_r)x^l\\
&=\sum_{l=0}^n\left(\sum_{j=l}^n\sum_{m=0}^{n-j}(-1)^{n-m-j}\binom{n}{j}\binom{j}{l}\frac{m!}{(m+1)^k}S_2(n-j,m)B_{j-l}(a_1,\dots,a_r)\right)x^l\,,  
\end{align*} 
which is the identity (\ref{100d}). 


By (\ref{uc16}) with (\ref{10c}), we have 
\begin{align*}  
\ang{g\bigl(\bar f(t)\bigr)^{-1}\bar f(t)^j|x^n}&=\ang{\frac{t^r}{\prod_{j=1}^r(e^{a_j t}-1)}\frac{{\rm Li}_k(1-e^{-t})}{1-e^{-t}}t^j\Big|x^n}\\
&=(n)_j\ang{\frac{t^r}{\prod_{j=1}^r(e^{a_j t}-1)}\frac{{\rm Li}_k(1-e^{-t})}{1-e^{-t}}\Big|x^{n-j}}\\
&=(n)_j\ang{\sum_{i=0}^\infty S_i^{(r,k)}(a_1,\dots,a_r)\frac{t^i}{i!}\Big|x^{n-j}}\\
&=(n)_j S_{n-j}^{(r,k)}(a_1,\dots,a_r)\,.   
\end{align*} 
Thus, we get (\ref{100e}).  
\qed\end{proof} 


\subsection{Sheffer identity} 

\begin{theorem} 
\begin{equation} 
S_n^{(r,k)}(x+y|a_1,\dots,a_r)=\sum_{j=0}^n\binom{n}{j}S_j^{(r,k)}(x|a_1,\dots,a_r)y^{n-j}\,. 
\label{20} 
\end{equation} 
\label{th200} 
\end{theorem} 

\begin{proof} 
By (\ref{10c}) with 
\begin{align*}  
p_n(x)&=\frac{\prod_{j=1}^r(e^{a_j t}-1)}{t^r}\frac{1-e^{-t}}{{\rm Li}_k(1-e^{-t})}S_n^{(r,k)}(x|a_1,\dots,a_r)\\
&=x^n\sim(1,t)\,,  
\end{align*}  
using (\ref{uc17}), we have (\ref{20}).  
\qed\end{proof}


\subsection{Recurrence}  

\begin{theorem} 
\begin{align}  
&S_{n+1}^{(r,k)}(x|a_1,\dots,a_r)=x S_n^{(r,k)}(x|a_1,\dots,a_r)\notag\\
&\quad -\frac{1}{n+1}\sum_{j=1}^r\sum_{l=0}^n\binom{n+1}{l}(-a_j)^{n+1-l}B_{n+1-l}S_l^{(r,k)}(x|a_1,\dots,a_r)\notag\\
&\quad -\frac{1}{n+1}\left(S_{n+1}^{(r+1,k)}(x|a_1,\dots,a_r,1)-S_{n+1}^{(r+1,k-1)}(x|a_1,\dots,a_r,1)\right)\,, \label{30} 
\end{align}  
where $B_n$ is the $n$th ordinary Bernoulli number.  
\label{th30} 
\end{theorem} 
\begin{proof} 
By applying  
$$
s_{n+1}(x)=\left(x-\frac{g'(t)}{g(t)}\right)\frac{1}{f'(t)}s_n(x)
$$ 
(\cite[Corollary 3.7.2]{Roman}) with (\ref{10c}), we get 
$$
S_{n+1}^{(r,k)}(x|a_1,\dots,a_r)=\left(x-\frac{g_{r,k}'(t)}{g_{r,k}(t)}\right)S_n^{(r,k)}(x|a_1,\dots,a_r)\,. 
$$   
Now,  
\begin{align*} 
&\frac{g_{r,k}'(t)}{g_{r,k}(t)}=(\ln g_{r,k}(t))'\\
&=\left(\sum_{j=1}^r\ln(e^{a_j t}-1)-r\ln t+\ln(1-e^{-t})-\ln{\rm Li}_k(1-e^{-t})\right)'\\
&=\sum_{j=1}^r\frac{a_j e^{a_j t}}{e^{a_j t}-1}-\frac{r}{t}+\frac{e^{-t}}{1-e^{-t}}\left(1-\frac{{\rm Li}_{k-1}(1-e^{-t})}{{\rm Li}_k(1-e^{-t})}\right)\\
&=\frac{\sum_{j=1}^r\prod_{i\ne j}(e^{a_i t}-1)(a_j t e^{a_j t}-e^{a_j t}+1)}{t\prod_{j=1}^r(e^{a_j t}-1)}+\frac{t}{e^t-1}\frac{{\rm Li}_k(1-e^{-t})-{\rm Li}_{k-1}(1-e^{-t})}{t{\rm Li}_k(1-e^{-t})}\,.  
\end{align*} 
Since 
\begin{align*}  
\frac{\sum_{j=1}^r\prod_{i\ne j}(e^{a_i t}-1)(a_j t e^{a_j t}-e^{a_j t}+1)}{\prod_{j=1}^r(e^{a_j t}-1)}&=\frac{\frac{1}{2}\bigl(\sum_{j=1}^r a_1\cdots a_{j-1}a_j^2 a_{j+1}\cdots a_r\bigr)t^{r+1}+\cdots}{(a_1\cdots a_r)t^r+\cdots}\\
&=\frac{1}{2}\left(\sum_{j=1}^r a_j\right)t+\cdots
\end{align*} 
is a series with order$\ge 1$ and 
$$ 
\frac{{\rm Li}_k(1-e^{-t})-{\rm Li}_{k-1}(1-e^{-t})}{1-e^{-t}}
=\left(\frac{1}{2^k}-\frac{1}{2^{k-1}}\right)t+\cdots
$$ 
is a delta series, we have 
\begin{align*}  
&S_{n+1}^{(r,k)}(x|a_1,\dots,a_r)
=x S_n^{(r,k)}(x|a_1,\dots,a_r)-\frac{g_{r,k}'(t)}{g_{r,k}(t)}S_n^{(r,k)}(x|a_1,\dots,a_r)\\
&=x S_n^{(r,k)}(x|a_1,\dots,a_r)-\frac{g_{r,k}'(t)}{g_{r,k}(t)}\frac{t^r}{\prod_{j=1}^r(e^{a_j t}-1)}\frac{{\rm Li}_k(1-e^{-t})}{1-e^{-t}}x^n\\
&=x S_n^{(r,k)}(x|a_1,\dots,a_r)\\
&\quad -\frac{t^r}{\prod_{j=1}^r(e^{a_j t}-1)}\frac{{\rm Li}_k(1-e^{-t})}{1-e^{-t}}\frac{\sum_{j=1}^r\prod_{i\ne j}(e^{a_i t}-1)(a_j t e^{a_j t}-e^{a_j t}+1)}{t\prod_{j=1}^r(e^{a_j t}-1)}x^n\\
&\quad -\frac{t^r}{\prod_{j=1}^r(e^{a_j t}-1)}\frac{t}{e^t-1}\frac{{\rm Li}_k(1-e^{-t})-{\rm Li}_{k-1}(1-e^{-t})}{t(1-e^{-t})}x^n\,. 
\end{align*} 
Now, 
\begin{align*} 
&\frac{\sum_{j=1}^r\prod_{i\ne j}(e^{a_i t}-1)(a_j t e^{a_j t}-e^{a_j t}+1)}{t\prod_{j=1}^r(e^{a_j t}-1)}x^n\\
&=\frac{\sum_{j=1}^r\prod_{i\ne j}(e^{a_i t}-1)(a_j t e^{a_j t}-e^{a_j t}+1)}{\prod_{j=1}^r(e^{a_j t}-1)}\frac{x^{n+1}}{n+1}\\
&=\frac{1}{n+1}\sum_{j=1}^r\frac{a_j t e^{a_j t}-e^{a_j t}+1}{e^{a_j t}-1}x^{n+1}\\
&=\frac{1}{n+1}\sum_{j=1}^r\left(\frac{a_j t e^{a_j t}}{e^{a_j t}-1}-1\right)x^{n+1}\\
&=\frac{1}{n+1}\sum_{j=1}^r\left(\sum_{l=0}^\infty\frac{(-1)^l B_l a_j^l}{l!}t^l-1\right)x^{n+1}\\
&=\frac{1}{n+1}\sum_{j=1}^r\left(\sum_{l=0}^{n+1}\binom{n+1}{l}(-a_j)^l B_l x^{n+1-l}-x^{n+1}\right)\\
&=\frac{1}{n+1}\sum_{j=1}^r\sum_{l=1}^{n+1}\binom{n+1}{l}(-a_j)^l B_l x^{n+1-l}\\
&=\frac{1}{n+1}\sum_{j=1}^r\sum_{l=0}^{n}\binom{n+1}{l}(-a_j)^{n+1-l}B_{n+1-l}x^l\,. 
\end{align*}
Also,   
$$   
\frac{{\rm Li}_k(1-e^{-t})-{\rm Li}_{k-1}(1-e^{-t})}{t(1-e^{-t})}x^n
=\frac{1}{n+1}\frac{{\rm Li}_k(1-e^{-t})-{\rm Li}_{k-1}(1-e^{-t})}{1-e^{-t}}x^{n+1}\,. 
$$ 
Thus, we get the identity (\ref{30}).  
\qed\end{proof} 


\subsection{A more relation} 

\begin{theorem} 
\begin{align} 
&S_n^{(r,k)}(x|a_1,\dots,a_r)=x S_{n-1}^{(r,k)}(x|a_1,\dots,a_r)\notag\\
&\quad +\sum_{m=1}^n\frac{(-1)^{m-1}\binom{n-1}{m-1}B_m}{m}\sum_{j=1}^r a_j^m S_{n-m}^{(r,k)}(x|a_1,\dots,a_r)\notag\\
&\quad +\frac{1}{n}S_n^{(r+1,k-1)}(x|a_1,\dots,a_r,1)-\frac{1}{n}S_n^{(r+1,k)}(x|a_1,\dots,a_r,1)\,. 
\label{40} 
\end{align} 
\label{th40} 
\end{theorem} 
\begin{proof} 
For $n\ge 1$ we have 
\begin{align*}  
S_n^{(r,k)}(y|a_1,\dots,a_r)&=\ang{\sum_{l=0}^\infty S_l^{(r,k)}(y|a_1,\dots,a_r)\frac{t^l}{l!}\Big|x^n}\\
&=\ang{\frac{t^r}{\prod_{j=1}^r(e^{a_j t}-1)}\frac{{\rm Li}_k(1-e^{-t})}{1-e^{-t}}e^{y t}\Big|x^n}\\
&=\ang{\partial_t\left(\frac{t^r}{\prod_{j=1}^r(e^{a_j t}-1)}\frac{{\rm Li}_k(1-e^{-t})}{1-e^{-t}}e^{y t}\right)\Big|x^{n-1}}\\
&=\ang{\left(\partial_t\frac{t^r}{\prod_{j=1}^r(e^{a_j t}-1)}\right)\frac{{\rm Li}_k(1-e^{-t})}{1-e^{-t}}e^{y t}\Big|x^{n-1}}\\
&\quad +\ang{\frac{t^r}{\prod_{j=1}^r(e^{a_j t}-1)}\left(\partial_t\frac{{\rm Li}_k(1-e^{-t})}{1-e^{-t}}\right)e^{y t}\Big|x^{n-1}}\\
&\quad +\ang{\frac{t^r}{\prod_{j=1}^r(e^{a_j t}-1)}\frac{{\rm Li}_k(1-e^{-t})}{1-e^{-t}}(\partial_t e^{y t})\Big|x^{n-1}}\\
&=y S_{n-1}^{(r,k)}(y|a_1,\dots,a_r)\\
&\quad +\ang{\left(\partial_t\frac{t^r}{\prod_{j=1}^r(e^{a_j t}-1)}\right)\frac{{\rm Li}_k(1-e^{-t})}{1-e^{-t}}e^{y t}\Big|x^{n-1}}\\
&\quad +\ang{\frac{t^r}{\prod_{j=1}^r(e^{a_j t}-1)}\left(\partial_t\frac{{\rm Li}_k(1-e^{-t})}{1-e^{-t}}\right)e^{y t}\Big|x^{n-1}}\,. 
\end{align*} 
Observe that 
\begin{align*} 
\partial_t\left(\frac{t^r}{\prod_{j=1}^r(e^{a_j t}-1)}\right)&=\frac{r t^{r-1}-t^r\sum_{j=1}^r\frac{a_j e^{a_j t}}{e^{a_j t}-1}}{\prod_{j=1}^r(e^{a_j t}-1)}\\
&=\frac{t^{r-1}}{\prod_{j=1}^r(e^{a_j t}-1)}\left(r-\sum_{j=1}^r\frac{a_j t e^{a_j t}}{e^{a_j t}-1}\right)\\
&=\frac{t^{r-1}}{\prod_{j=1}^r(e^{a_j t}-1)}\left(r-\sum_{j=1}^r\frac{-a_j t}{e^{-a_j t}-1}\right)\\
&=\frac{t^{r-1}}{\prod_{j=1}^r(e^{a_j t}-1)}\left(r-\sum_{j=1}^r\sum_{m=0}^\infty\frac{(-a_j)^m B_m t^m}{m!}\right)\\
&=\frac{t^{r-1}}{\prod_{j=1}^r(e^{a_j t}-1)}\left(r-\sum_{m=0}^\infty\left(\sum_{j=1}^r(-a_j)^m\right)\frac{B_m t^m}{m!}\right)\\
&=\frac{t^{r}}{\prod_{j=1}^r(e^{a_j t}-1)}\sum_{m=1}^\infty\left(\sum_{j=1}^r a_j^m\right)\frac{(-1)^{m-1}B_m}{m!}t^{m-1}\,. 
\end{align*} 
Thus,  
\begin{align*}  
&\ang{\left(\partial_t\frac{t^r}{\prod_{j=1}^r(e^{a_j t}-1)}\right)\frac{{\rm Li}_k(1-e^{-t})}{1-e^{-t}}e^{y t}\Big|x^{n-1}}\\
&=\ang{\frac{t^r}{\prod_{j=1}^r(e^{a_j t}-1)}\frac{{\rm Li}_k(1-e^{-t})}{1-e^{-t}}e^{y t}\Big|\sum_{m=1}^n\left(\sum_{j=1}^r a_j^m\right)\frac{(-1)^{m-1}B_m}{m!}t^{m-1}x^{n-1}}\\
&=\sum_{m=1}^n\frac{(-1)^{m-1}\binom{n-1}{m-1}B_m}{m}\sum_{j=1}^r a_j^m\ang{\frac{t^r}{\prod_{j=1}^r(e^{a_j t}-1)}\frac{{\rm Li}_k(1-e^{-t})}{1-e^{-t}}e^{y t}\Big|x^{n-m}}\\
&=\sum_{m=1}^n\frac{(-1)^{m-1}\binom{n-1}{m-1}B_m}{m}S_{n-m}^{(r,k)}(y|a_1,\dots,a_r)\sum_{j=1}^r a_j^m\,. 
\end{align*} 
Since 
$$ 
\frac{{\rm Li}_{k-1}(1-e^{-t})-{\rm Li}_k(1-e^{-t})}{1-e^{-t}}
=\left(\frac{1}{2^{k-1}}-\frac{1}{2^k}\right)t+\cdots
$$ 
is a delta series, we have  
\begin{align*} 
&\ang{\frac{t^r}{\prod_{j=1}^r(e^{a_j t}-1)}\left(\partial_t\frac{{\rm Li}_k(1-e^{-t})}{1-e^{-t}}\right)e^{y t}\Big|x^{n-1}}\\
&=\ang{\frac{t^r}{\prod_{j=1}^r(e^{a_j t}-1)}\frac{e^{-t}\bigl({\rm Li}_{k-1}(1-e^{-t})-{\rm Li}_k(1-e^{-t})\bigr)}{(1-e^{-t})^2}e^{y t}\Big|x^{n-1}}\\
&=\ang{\frac{t^r}{\prod_{j=1}^r(e^{a_j t}-1)}\frac{t}{e^t-1}\frac{{\rm Li}_{k-1}(1-e^{-t})-{\rm Li}_k(1-e^{-t})}{t(1-e^{-t})}e^{y t}\Big|x^{n-1}}\\
&=\ang{\frac{t^{r+1}}{\prod_{j=1}^r(e^{a_j t}-1)(e^t-1)}\frac{{\rm Li}_{k-1}(1-e^{-t})-{\rm Li}_k(1-e^{-t})}{1-e^{-t}}e^{y t}\Big|\frac{x^n}{n}}\\
&=\frac{1}{n}\ang{\frac{t^{r+1}}{\prod_{j=1}^r(e^{a_j t}-1)(e^t-1)}\frac{{\rm Li}_{k-1}(1-e^{-t})}{1-e^{-t}}e^{y t}\Big|x^n}\\
&\quad -\frac{1}{n}\ang{\frac{t^{r+1}}{\prod_{j=1}^r(e^{a_j t}-1)(e^t-1)}\frac{{\rm Li}_{k}(1-e^{-t})}{1-e^{-t}}e^{y t}\Big|x^n}\\
&=\frac{1}{n}S_n^{(r+1,k-1)}(y|a_1,\dots,a_r,1)-\frac{1}{n}S_n^{(r+1,k)}(y|a_1,\dots,a_r,1)\,. 
\end{align*}  
Therefore, we obtain the desired result.  
\qed\end{proof}

\noindent 
{\it Remark.}  
After simple modification, Theorem \ref{th40} becomes 
\begin{align*} 
&S_{n+1}^{(r,k)}(x|a_1,\dots,a_r)=x S_{n}^{(r,k)}(x|a_1,\dots,a_r)\\
&\quad +\sum_{l=1}^{n+1}\frac{(-1)^{l-1}\binom{n}{l-1}B_l}{l}\sum_{j=1}^r a_j^l S_{n+1-l}^{(r,k)}(x|a_1,\dots,a_r)\\
&\quad +\frac{1}{n+1}S_{n+1}^{(r+1,k-1)}(x|a_1,\dots,a_r,1)-\frac{1}{n+1}S_{n+1}^{(r+1,k)}(x|a_1,\dots,a_r,1)\,. 
\end{align*} 
which is the same as the above recurrence formula (\ref{30}) upon replacing $n$ by $n-1$.


\subsection{Relations with poly-Bernoulli numbers and Barnes' multiple Bernoulli numbers} 

\begin{theorem} 
\begin{multline} 
\sum_{m=0}^n\binom{n+1}{m}(-1)^{n-m}S_m^{(r,k)}(a_1,\dots,a_r)\\
=\sum_{l=0}^n\sum_{m=0}^l(-1)^{l-m}\binom{l}{m}\binom{n+1}{l+1}B_m^{(k-1)}B_{n-l}(a_1,\dots,a_r)\,. 
\label{50} 
\end{multline} 
\label{th50} 
\end{theorem} 
\begin{proof} 
We shall compute 
$$
\ang{\frac{t^r}{\prod_{j=1}^r(e^{a_j t}-1)}{\rm Li}_k(1-e^{-t})\Big|x^{n+1}}
$$ 
in two different ways. 
On the one hand,  
\begin{align*} 
&\ang{\frac{t^r}{\prod_{j=1}^r(e^{a_j t}-1)}{\rm Li}_k(1-e^{-t})\Big|x^{n+1}}
=\ang{\frac{t^r}{\prod_{j=1}^r(e^{a_j t}-1)}\frac{{\rm Li}_k(1-e^{-t})}{1-e^{-t}}\Big|(1-e^{-t})x^{n+1}}\\
&=\ang{\frac{t^r}{\prod_{j=1}^r(e^{a_j t}-1)}\frac{{\rm Li}_k(1-e^{-t})}{1-e^{-t}}\Big|x^{n+1}-(x-1)^{n+1}}\\
&=\sum_{m=0}^n\binom{n+1}{m}(-1)^{n-m}\ang{\frac{t^r}{\prod_{j=1}^r(e^{a_j t}-1)}\frac{{\rm Li}_k(1-e^{-t})}{1-e^{-t}}\Big|x^m}\\
&=\sum_{m=0}^n\binom{n+1}{m}(-1)^{n-m}S_m^{(r,k)}(a_1,\dots,a_r)\,. 
\end{align*} 
On the other hand,  
\begin{align*} 
&\ang{\frac{t^r}{\prod_{j=1}^r(e^{a_j t}-1)}{\rm Li}_k(1-e^{-t})\Big|x^{n+1}}
=\ang{{\rm Li}_k(1-e^{-t})\Big|\frac{t^r}{\prod_{j=1}^r(e^{a_j t}-1)}x^{n+1}}\\
&=\ang{{\rm Li}_k(1-e^{-t})\Big|B_{n+1}(x|a_1,\dots,a_r)}\\
&=\ang{\int_0^t\bigl({\rm Li}_k(1-e^{-s})\bigr)'ds\Big|B_{n+1}(x|a_1,\dots,a_r)}\\
&=\ang{\int_0^t e^{-s}\frac{{\rm Li}_{k-1}(1-e^{-s})}{1-e^{-s}}ds\Big|B_{n+1}(x|a_1,\dots,a_r)}\\
&=\ang{\int_0^t\left(\sum_{j=0}^\infty\frac{(-s)^j}{j!}\right)\left(\sum_{m=0}^\infty\frac{B_m^{(k-1)}}{m!}s^m\right)ds\Big|B_{n+1}(x|a_1,\dots,a_r)}\\
&=\ang{\sum_{l=0}^\infty\left(\sum_{m=0}^l(-1)^{l-m}\binom{l}{m}B_m^{(k-1)}\right)\frac{1}{l!}\int_0^ts^l ds\Big|B_{n+1}(x|a_1,\dots,a_r)}\\
&=\sum_{l=0}^n\sum_{m=0}^l(-1)^{l-m}\binom{l}{m}\frac{B_m^{(k-1)}}{(l+1)!}\ang{t^{l+1}\Big|B_{n+1}(x|a_1,\dots,a_r)}\\
&=\sum_{l=0}^n\sum_{m=0}^l(-1)^{l-m}\binom{l}{m}\frac{B_m^{(k-1)}}{(l+1)!}(n+1)_{l+1}B_{n-l}(a_1,\dots,a_r)\\
&=\sum_{l=0}^n\sum_{m=0}^l(-1)^{l-m}\binom{l}{m}\binom{n+1}{l+1}B_m^{(k-1)}B_{n-l}(a_1,\dots,a_r)\,. 
\end{align*} 
Here, $B_{n-l}(a_1,\dots,a_r)=B_{n-l}(0|a_1,\dots,a_r)$. 
Thus, we get (\ref{50}).  
\qed\end{proof}


\subsection{Relations with the Stirling numbers of the second kind and the falling factorials}  

\begin{theorem} 
\begin{equation}  
S_n^{(r,k)}(x|a_1,\dots,a_r)
=\sum_{m=0}^n\left(\sum_{l=m}^n S_2(l,m)\binom{n}{l}S_{n-l}^{(r,k)}(a_1,\dots,a_r)\right)(x)_m\,. 
\label{60} 
\end{equation} 
\label{th60} 
\end{theorem} 
\begin{proof} 
For (\ref{10c}) and $(x)_n\sim(1,e^t-1)$, assume that 
$S_n^{(r,k)}(x|a_1,\dots,a_r)=\sum_{m=0}^n C_{n,m}(x)_m$. By (\ref{uc19}), we have 
\begin{align*} 
C_{n,m}&=\frac{1}{m!}\ang{\frac{1}{\frac{\prod_{j=1}^r(e^{a_j t}-1)}{t^r}\frac{1-e^{-t}}{{\rm Li}_k(1-e^{-t})}}(e^t-1)^m\Big|x^n}\\
&=\frac{1}{m!}\ang{\frac{t^r}{\prod_{j=1}^r(e^{a_j t}-1)}\frac{{\rm Li}_k(1-e^{-t})}{1-e^{-t}}\Big|(e^t-1)^m x^n}\\
&=\frac{1}{m!}\ang{\frac{t^r}{\prod_{j=1}^r(e^{a_j t}-1)}\frac{{\rm Li}_k(1-e^{-t})}{1-e^{-t}}\Big|m!\sum_{l=m}^n S_2(l,m)\frac{t^l}{l!}x^n}\\
&=\sum_{l=m}^n S_2(l,m)\binom{n}{l}\ang{\frac{t^r}{\prod_{j=1}^r(e^{a_j t}-1)}\frac{{\rm Li}_k(1-e^{-t})}{1-e^{-t}}\Big|x^{n-l}}\\
&=\sum_{l=m}^n S_2(l,m)\binom{n}{l}S_{n-l}^{(r,k)}(a_1,\dots,a_r)\,. 
\end{align*} 
Thus, we get the identity (\ref{60}). 
\qed\end{proof}


\subsection{Relations with the Stirling numbers of the second kind and the rising factorials} 

\begin{theorem} 
\begin{equation} 
S_n^{(r,k)}(x|a_1,\dots,a_r)=\sum_{m=0}^n\left(\sum_{l=m}^n S_2(l,m)\binom{n}{l}S_{n-l}^{(r,k)}(-m|a_1,\dots,a_r)\right)(x)^{(m)}\,. 
\label{70} 
\end{equation} 
\label{th70} 
\end{theorem} 
\begin{proof} 
For (\ref{10c}) and $(x)^{(n)}=x(x+1)\cdots(x+n-1)\sim(1,1-e^{-t})$, assume that 
$S_n^{(r,k)}(x|a_1,\dots,a_r)=\sum_{m=0}^n C_{n,m}(x)^{(m)}$. By (\ref{uc19}), we have 
\begin{align*} 
C_{n,m}&=\frac{1}{m!}\ang{\frac{1}{\frac{\prod_{j=1}^r(e^{a_j t}-1)}{t^r}\frac{1-e^{-t}}{{\rm Li}_k(1-e^{-t})}}(1-e^{-t})^m\Big|x^n)}\\
&=\frac{1}{m!}\ang{\frac{t^r}{\prod_{j=1}^r(e^{a_j t}-1)}\frac{{\rm Li}_k(1-e^{-t})}{1-e^{-t}}e^{-m t}\Big|(e^t-1)^m x^n)}\\
&=\sum_{l=m}^n S_2(l,m)\binom{n}{l}\ang{e^{-mt}\Big|\frac{t^r}{\prod_{j=1}^r(e^{a_j t}-1)}\frac{{\rm Li}_k(1-e^{-t})}{1-e^{-t}}x^{n-l}}\\
&=\sum_{l=m}^n S_2(l,m)\binom{n}{l}\ang{e^{-mt}\Big|S_{n-l}^{(r,k)}(x|a_1,\dots,a_r)}\\
&=\sum_{l=m}^n S_2(l,m)\binom{n}{l}S_{n-l}^{(r,k)}(-m|a_1,\dots,a_r)\,. 
\end{align*}
Thus, we get the identity (\ref{70}). 
\qed\end{proof}


\subsection{Relations with higher-order Frobenius-Euler polynomials} 

For $\lambda\in\mathbb C$ with $\lambda\ne 1$, the Frobenius-Euler polynomials of order $r$, $H_{n}^{(r)}(x|\lambda)$ are defined by the generating function 
$$
\left(\frac{1-\lambda}{e^t-\lambda}\right)^r e^{xt}=\sum_{n=0}^\infty H_{n}^{(r)}(x|\lambda)\frac{t^n}{n!}
$$ 
(see e.g. \cite{KimKim1}). 

\begin{theorem} 
\begin{equation} 
S_n^{(r,k)}(x|a_1,\dots,a_r)
=\sum_{m=0}^n\left(\frac{\binom{n}{m}}{(1-\lambda)^s}\sum_{j=0}^s\binom{s}{j}(-\lambda)^{s-j}S_{n-m}^{(r,k)}(j|a_1,\dots,a_r)\right)H_m^{(s)}(x|\lambda)\,. 
\label{80} 
\end{equation} 
\label{th80} 
\end{theorem} 
\begin{proof} 
For (\ref{10c}) and 
$$
H_n^{(s)}(x|\lambda)\sim\left(\left(\frac{e^t-\lambda}{1-\lambda}\right)^s, t\right)\,,  
$$   
assume that 
$S_n^{(r,k)}(x|a_1,\dots,a_r)=\sum_{m=0}^n C_{n,m}H_m^{(s)}(x|\lambda)$. By (\ref{uc19}), we have 
\begin{align*} 
C_{n,m}&=\frac{1}{m!}\ang{\left(\frac{e^t-\lambda}{1-\lambda}\right)^s\frac{t^r}{\prod_{j=1}^r(e^{a_j t}-1)}\frac{{\rm Li}_k(1-e^{-t})}{1-e^{-t}}t^m\Big|x^n}\\
&=\frac{1}{m!(1-\lambda)^s}\ang{(e^t-\lambda)^s\frac{t^r}{\prod_{j=1}^r(e^{a_j t}-1)}\frac{{\rm Li}_k(1-e^{-t})}{1-e^{-t}}\Big|t^m x^n}\\
&=\frac{\binom{n}{m}}{(1-\lambda)^s}\sum_{j=0}^s\binom{s}{j}(-\lambda)^{s-j}\ang{e^{j t}\Big|\frac{t^r}{\prod_{j=1}^r(e^{a_j t}-1)}\frac{{\rm Li}_k(1-e^{-t})}{1-e^{-t}}x^{n-m}}\\
&=\frac{\binom{n}{m}}{(1-\lambda)^s}\sum_{j=0}^s\binom{s}{j}(-\lambda)^{s-j}S_{n-m}^{(r,k)}(j|a_1,\dots,a_r)\,. 
\end{align*} 
Thus, we get the identity (\ref{80}).  
\qed\end{proof}


\subsection{Relations with higher-order Bernoulli polynomials} 

Bernoulli polynomials $\mathfrak B_n^{(r)}(x)$ of order $r$ are defined by 
$$  
\left(\frac{t}{e^t-1}\right)^r e^{xt}=\sum_{n=0}^\infty\frac{\mathfrak B_n^{(r)}(x)}{n!}t^n
$$   
(see e.g. \cite[Section 2.2]{Roman}). 

\begin{theorem} 
\begin{equation} 
S_n^{(r,k)}(x|a_1,\dots,a_r)
=\sum_{m=0}^n\left(\binom{n}{m}\sum_{l=0}^{n-m}\frac{\binom{n-m}{l}}{\binom{l+s}{l}}S_2(l+s,s)S_{n-m-l}^{(r,k)}(a_1,\dots,a_r)\right)\mathfrak B_m^{(s)}(x)\,. 
\label{90} 
\end{equation} 
\label{th90} 
\end{theorem} 
\begin{proof} 
For (\ref{10c}) and 
$$
\mathfrak B_n^{(s)}(x)\sim\left(\left(\frac{e^t-1}{t}\right)^s, t\right)\,,  
$$  
assume that $S_n^{(r,k)}(x|a_1,\dots,a_r)=\sum_{m=0}^n C_{n,m}\mathfrak B_m^{(s)}(x)$.  
By (\ref{uc19}), we have  
\begin{align*} 
C_{n,m}&=\frac{1}{m!}\ang{\left(\frac{e^t-1}{t}\right)^s\frac{t^r}{\prod_{j=1}^r(e^{a_j t}-1)}\frac{{\rm Li}_k(1-e^{-t})}{1-e^{-t}}t^m\Big|x^n}\\
&=\binom{n}{m}\ang{\frac{t^r}{\prod_{j=1}^r(e^{a_j t}-1)}\frac{{\rm Li}_k(1-e^{-t})}{1-e^{-t}}\Big|\left(\frac{e^t-1}{t}\right)^s x^{n-m}}\\
&=\binom{n}{m}\ang{\frac{t^r}{\prod_{j=1}^r(e^{a_j t}-1)}\frac{{\rm Li}_k(1-e^{-t})}{1-e^{-t}}\Big|\sum_{l=0}^{n-m}\frac{s!}{(l+s)!}S_2(l+s,s)t^l x^{n-m}}\\
&=\binom{n}{m}\sum_{l=0}^{n-m}\frac{s!}{(l+s)!}S_2(l+s,s)(n-m)_l\ang{\frac{t^r}{\prod_{j=1}^r(e^{a_j t}-1)}\frac{{\rm Li}_k(1-e^{-t})}{1-e^{-t}}\Big|x^{n-m-l}}\\
&=\binom{n}{m}\sum_{l=0}^{n-m}\frac{s!}{(l+s)!}S_2(l+s,s)(n-m)_l S_{n-m-l}^{(r,k)}(a_1,\dots,a_r)\\
&=\binom{n}{m}\sum_{l=0}^{n-m}\frac{\binom{n-m}{l}}{\binom{l+s}{l}}S_2(l+s,s)S_{n-m-l}^{(r,k)}(a_1,\dots,a_r)\,. 
\end{align*} 
Thus, we get the identity (\ref{90}).  
\qed\end{proof}

\end{document}